\documentclass{amsart}
\usepackage{latexsym}
\usepackage{amstext}
\usepackage{amsmath}
\usepackage{amssymb}
\usepackage{amsthm}
\usepackage{url}
\usepackage{enumerate}

\newtheorem{theorem}{Theorem}
\newtheorem{lemma}[theorem]{Lemma}
\newtheorem{corollary}[theorem]{Corollary}
\newtheorem{conjecture}[theorem]{Conjecture}

\theoremstyle{definition}

\newtheorem{remark}[theorem]{Remark}
\newtheorem{example}[theorem]{Example}

\newtheorem*{statementA}{Statement A}
\newtheorem*{statementB}{Statement B}
\newtheorem*{statementA'}{Statement A'}

\DeclareMathOperator{\ord}{ord}

\DeclareMathOperator{\Gal}{Gal}
\DeclareMathOperator{\tors}{tors}
\def \co {\mathcal{O}}

\def \pre {\rm{pre}}
\def \pr {{\rm Preim}(\phi,a,K)}
\def \rat {{\rm Rat}_d}
\def \prp {{\rm PrePer}(\phi,K)}

\def \Kbar {\overline{K}}
\def \Lbar {\overline{L}}
\def \kbar {\overline{k}}
\def \Qbar {\overline{\mathbb{Q}}}

\DeclareMathOperator{\Res}{Res}
\DeclareMathOperator{\dv}{div}
\DeclareMathOperator{\M}{M}
\DeclareMathOperator{\N}{N}
\DeclareMathOperator{\Pic}{Pic}

\begin{document}
\bibliographystyle{amsplain}
\title{Rational preimages in families of dynamical systems}
\author{Aaron Levin}
\address{Department of Mathematics\\Michigan State University\\East Lansing, MI 48824}
\curraddr{}
\email{adlevin@math.msu.edu}
\date{}
\begin{abstract}
Let $\phi$ be a rational function of degree at least two defined over a number field $k$.  Let $a\in \mathbb{P}^1(k)$ and let $K$ be a number field containing $k$.  We study the cardinality of the set of rational iterated preimages
\begin{equation*}
\pr=\{x_0\in \mathbb{P}^1(K)\mid \phi^N(x_0)=a \text{ for some }N\geq 1\}.
\end{equation*}
We prove two new results (Theorem \ref{m1} and Theorem \ref{m2}) bounding $|\pr|$ as $\phi$ varies in certain families of rational functions.


Our proofs are based on unit equations and a method of Runge for effectively determining integral points on certain affine curves.  We also formulate and state a uniform boundedness conjecture for $\pr$ and relate this conjecture to other well-known conjectures in arithmetic dynamics.
\end{abstract}
\maketitle

\section{Introduction}
Let $\phi\in K(x)$ be a rational function of degree $d\geq 2$, where $K$ is a number field, and let $a\in K\cup \{\infty\}=\mathbb{P}^1(K)$.  We will be interested in the {\it set of  $K$-rational iterated preimages of $a$}:
\begin{equation*}
\pr=\{x_0\in \mathbb{P}^1(K)\mid \phi^N(x_0)=a \text{ for some }N\geq 1\}.
\end{equation*}
Here $\phi^N$ denotes the $N$th iterate of $\phi$.  By standard properties of the canonical height associated to $\phi$, the set $\pr$ is known to be finite.  In fact, if $\phi$ and $a$ are fixed, the cardinality can be bounded solely in terms of $[K:\mathbb{Q}]$.  A natural question that arises is how the cardinality $|\pr|$ depends on $\phi$, $a$, and $K$.

It is easy to see that $\pr$ depends on $a$ in a nontrivial way.  For any $x_0\in \mathbb{P}^1(K)$ that is not preperiodic for $\phi$, and any $n\in \mathbb{N}$, we have trivially $|{\rm Preim}(\phi,\phi^n(x_0),K)|\geq n$.  In particular, $|\pr|$ is unbounded as we vary over $a\in \mathbb{P}^1(K)$.  Here, we will primarily be interested in fixing $a$ and looking at the dependence of $|\pr|$ on $\phi$ and $K$, but we will also study cases where we allow $a$ to vary in an algebraic way with $\phi$.

In the next section we discuss the following uniform boundedness conjecture for iterated preimages of rational functions (see Section \ref{scon} for the relevant definitions and Conjecture \ref{cmain2} for a stronger conjecture). 
\begin{conjecture}
\label{cmain}
Let $\mathcal{F}$ be a simple family of rational functions of degree $d\geq 2$ defined over a number field $k$.  Let $a \in \mathbb{P}^1(k)$ and let $K\supset k$ be a number field of degree $D$ over $\mathbb{Q}$.  Then there exists an integer $\kappa(\mathcal{F},D,a)$ such that
\begin{equation*}
|\pr|\leq\kappa(\mathcal{F},D,a)
\end{equation*}
for every $\phi\in \mathcal{F}(K)$.
\end{conjecture}
The most important hypothesis here is that the family $\mathcal{F}$ be simple, which means that there are only finitely many rational functions in $\mathcal{F}(\kbar)$ in any given conjugacy class of rational functions (with respect to conjugation by linear fractional transformations over $\kbar$).  We will show in Section \ref{scon} that a weaker version of Conjecture \ref{cmain} (where $\kappa$ is allowed to also depend on $K$) is a consequence of a conjecture of Silverman on canonical heights of wandering points combined with a uniform boundedness conjecture of Silverman and Morton for preperiodic points.  The straight-forward higher-dimensional analogue of Conjecture \ref{cmain} for endomorphisms of $\mathbb{P}^n$ is also plausible.



As a small step towards Conjecture \ref{cmain}, we prove two results giving certain effective bounds for $|\pr|$ in various special cases.  First, in Section \ref{secunit}, we prove a general result for rational functions of the form
\begin{equation*}
\phi(x)=\frac{x^d+a_{d-1}x^{d-1}+\cdots+a_1x}{b_{d-1}x^{d-1}+b_{d-2}x^{d-2}+\cdots+b_1x+1}.
\end{equation*}
As usual, we let $\co_K$ denote the ring of integers of a number field $K$ and we let $\co_{K,S}$ denote the ring of $S$-integers of $K$, where $S$ is a finite set of places of $K$, which we always assume contains the archimedean places of $K$.

\begin{theorem}
\label{m1}
Let $d>1$ be an integer.  Let $K$ be a number field and let $S$ be a finite set of places of $K$ containing the archimedean places.  Let $\phi$ be a rational function of degree $d$ of the form
\begin{equation}
\label{m1eq}
\phi(x)=\frac{x^d+a_{d-1}x^{d-1}+\cdots+a_1x}{b_{d-1}x^{d-1}+b_{d-2}x^{d-2}+\cdots+b_1x+1},
\end{equation}
where $a_1,\ldots,a_{d-1},b_1,\ldots,b_{d-1}\in \co_{K,S}$.  Let $a\in K^*$.
Then there exists an effectively computable integer $\kappa(d, a,|S|)$, depending only on $d$, $a$, and $|S|$ such that
\begin{equation*}
|{\rm Preim}(\phi,a,K)|\leq \kappa(d,a,|S|).
\end{equation*}

If $\phi\neq x^d$ and $a\in \co_{K,S}^*$, then there exists an effectively computable bound also independent of $a$:
\begin{equation*}
|{\rm Preim}(\phi,a,K)|\leq \kappa'(d,|S|).
\end{equation*}
For instance, we may take $\kappa'(d,|S|)=\exp(\exp(d^{20d}|S|))$.
\end{theorem}
\noindent
The proof of Theorem \ref{m1} is based on a reduction to certain unit equations.

\begin{remark}
Theorem \ref{m1} can be used to give a bound for $|\pr|$ for {\it any} rational function $\phi$ and $a\in \mathbb{P}^1(K)$.  Indeed, in this case it's not hard to see that there exists a linear fractional transformation $L\in \Kbar(x)$ such that $L(a)\neq 0,\infty$ and $L\circ \phi\circ L^{-1}$ is a rational function of the form \eqref{m1eq}.  Let $K'\supset K$ be a number field that $L$ is defined over.  Then, since there is a bijection between the sets ${\rm Preim}(\phi,a,K')$ and ${\rm Preim}(L\circ \phi\circ L^{-1},L(a),K')$, we may use Theorem \ref{m1} to bound $|\pr|$.
\end{remark}

In the final section of the paper we study and prove results for certain one-parameter families of quadratic rational functions.

\begin{theorem}
\label{m2}
Let $K$ be a number field of degree $D$ over $\mathbb{Q}$.   For $t\in K$ let $s(t)$ be the number of primes $\mathfrak{p}$ of $K$ for which $t$ is non-integral (i.e., $|t|_\mathfrak{p}>1$).  Let $a,b,c\in K[t]$ be polynomials.
\begin{enumerate}[(a)]
\item  
\label{pa}
Suppose that $b^2-4c-2b$ is nonconstant.  Let 
\begin{equation*}
\phi_t(x)=x^2+b(t)x+c(t).  
\end{equation*}
For $t\in K$ there exists an effectively computable integer $\kappa(D,a,b,c,s(t))$ such that
\begin{equation*}
|{\rm Preim}(\phi_t,a(t),K)|\leq \kappa(D,a,b,c,s(t)).
\end{equation*}

\item 
\label{pb}
Suppose that $bc$ is nonconstant and that if $c$ is constant then $a$ is also constant.  Let 
\begin{equation*}
\phi_t(x)=\frac{x^2+b(t)x}{c(t)x+1}.  
\end{equation*}
For $t\in K$ satisfying $b(t)c(t)\neq 1$ there exists an effectively computable integer $\kappa(D,a,b,c,s(t))$ such that
\begin{equation*}
|{\rm Preim}(\phi_t,a(t),K)|\leq \kappa(D,a,b,c,s(t)).
\end{equation*}
\end{enumerate}
\end{theorem}

Our proof of Theorem \ref{m2} is based on an old method of Runge \cite{Run} for effectively determining integral points on certain affine curves. 

\begin{remark}
In particular, Theorem \ref{m2} states that for these one-parameter families we have proved uniform bounds if one restricts to {\it integral} parameters, i.e., $t\in \co_{K,S}$ for some fixed finite set of places $S$ of $K$.
\end{remark}

\begin{remark}
The condition that $b^2-4c-2b$ be nonconstant in the first part of the theorem is necessary.  One may check that $b^2-4c-2b$ is constant if and only if there exists a linear fractional transformation $L\in \overline{K(t)}(x)$ such that we have $L\circ (x^2+b(t)x+c(t))\circ L^{-1}\in \overline{K}(x)$.  In this case, one easily sees that it may happen that $|{\rm Preim}(\phi_t,a(t),K)|$ is unbounded for, say, $t\in \co_K$.
\end{remark}

\begin{remark}
\label{rem7}
The case $a=\infty$ can easily be added to Theorem \ref{m2}.  Only part \eqref{pb} is nontrivial in this case and follows from the fact that (adding $b$ and $c$ into the notation)
\begin{equation*}
\phi_{b,c,t}^N(x)=\infty \Longleftrightarrow \phi_{c,b,t}^N\left(\frac{1}{x}\right)=0.
\end{equation*}
So $|{\rm Preim}(\phi_{b,c,t},\infty,K)|=|{\rm Preim}(\phi_{c,b,t},0,K)|$, and the latter cardinality is covered in Theorem \ref{m2}.
\end{remark}

We now discuss earlier results.  Conjecture \ref{cmain} is largely known when $\mathcal{F}$ is the family of quadratic polynomials $x^2+t$.

\begin{theorem}[Faber, Hutz, Ingram, Jones, Manes, Tucker, Zieve \cite{Fab}]
\label{tFab}
For all but finitely many values $a\in \Qbar$, if $K$ is a number field of degree $D$, 
\begin{equation*}
\mathcal{F}(K)=\{x^2+t\mid t\in K\}, 
\end{equation*}
and $a\in K$, then there exists an integer $\kappa(D,a)$ such that
\begin{equation*}
|\pr|\leq\kappa(D,a)
\end{equation*}
for every rational function $\phi\in \mathcal{F}(K)$.
\end{theorem}

The proof given in \cite{Fab} is ineffective and does not give an explicit upper bound $\kappa(D,a)$.  For the same family $\mathcal{F}$, in the special case $a=0$ and $K=\mathbb{Q}$, Faber, Hutz and Stoll \cite{Fab2} have shown, conditional on the Birch and Swinnerton-Dyer conjecture and related conjectures, that 
\begin{equation*}
\max_{\phi\in \mathcal{F}(\mathbb{Q})}|{\rm Preim}(\phi,0,\mathbb{Q})|=6.
\end{equation*}
Moreover, they also prove (unconditionally) that $|{\rm Preim}(\phi,0,\mathbb{Q})|\leq 6$ for all but finitely many $\phi\in \mathcal{F}(\mathbb{Q})$.

Along the same lines, for any number field $K$ and $a\in K$, Hutz, Hyde, and Krause \cite{HHK} have computed explicit sharp upper bounds $\bar{\kappa}(a,K)\in \{4,6,8,10\}$ such that $|\pr|\leq \bar{\kappa}(a,K)$ for all but finitely many $\phi\in \mathcal{F}(K)$.

As noted in \cite[Remark 4.9]{Fab}, results on canonical heights in \cite{In} imply an effective, but weaker, version of Theorem \ref{tFab} where, as in Theorem \ref{m2}, we allow $\kappa$ to also depend on the number of primes of $K$ at which $t$ is not integral.
\begin{theorem}[Ingram]
\label{tin}
Let $K$ be a number field of degree $D$ over $\mathbb{Q}$.  Let $a,t\in K$.  Let $s(t)$ be the number of primes $\mathfrak{p}$ of $K$ for which $|t|_\mathfrak{p}>1$.  Let $\phi_t=x^2+t$.  Then there exists an effectively computable integer $\kappa(D,a,s(t))$ such that
\begin{equation*}
|{\rm Preim}(\phi_t,a,K)|\leq\kappa(D,a,s(t)).
\end{equation*}
\end{theorem}

Note that part \eqref{pa} of Theorem \ref{m2} generalizes Theorem \ref{tin} to other one-parameter families of quadratic polynomials.  It seems likely that the method used in Theorem~\ref{m2} can successfully yield results for many other one-parameter families.


\section{The Uniform Boundedness Conjecture for Iterated Preimages}
\label{scon}

We first make a few definitions and recall some basic facts (see \cite{Sil2} or \cite[Ch. 4]{Sil} for a much more detailed treatment).  Let $d$ be a positive integer.  The set of rational functions of degree $d$ (over, say, $\Qbar$) can naturally be given the structure of an algebraic variety.  There is an affine variety defined over $\mathbb{Q}$, denoted $\rat$, such that for any number field $K$, the set of points $\rat(K)$ naturally corresponds to the set of rational functions over $K$ of degree $d$ .  Specifically, to a rational function $\phi$ of degree $d$,
\begin{equation*}
\phi=\frac{F_{\mathbf{a}}}{F_{\mathbf{b}}}=\frac{a_dx^d+a_{d-1}x^{d-1}+\cdots+a_1x+a_0}{b_dx^d+b_{d-1}x^{d-1}+\cdots+b_1x+b_0},
\end{equation*}
we can associate the point of projective space $[a_0,\ldots,a_d,b_0,\ldots,b_d]\in \mathbb{P}^{2d+1}$.  That $\phi$ has degree $d$ (and not smaller degree) is equivalent to the nonvanishing of the resultant $\rho=\Res(F_{\mathbf{a}},F_{\mathbf{b}})$ (which in turn is equivalent to the fact that $F_{\mathbf{a}}$ and $F_{\mathbf{b}}$ have no common zero (over $\Kbar$) and not both $a_d$ and $b_d$ are zero).  As $\rho$ is a homogeneous polynomial over $\mathbb{Z}$ in $a_0,\ldots,a_d,b_0,\ldots,b_d$, it's clear that, under the above correspondence, the affine variety $\rat=\mathbb{P}^{2d+1}\setminus \{\rho=0\}$ has the desired property.

The set of linear fractional transformations over $\Qbar$ acts on the set of degree $d$ rational functions $\rat(\Qbar)$ via conjugation.  The resulting conjugacy classes can again be naturally put into correspondence with the set of algebraic points of an affine variety, denoted $\M_d$.  Moreover, there is a morphism $\langle \cdot \rangle:\rat\to \M_d$ such that the fiber above a point $P\in \M_d(\Qbar)$ consists exactly of the set of rational functions in the conjugacy class associated to $P$.

By a {\it family $\mathcal{F}$ of rational functions of degree $d$ defined over a number field $K$} we will simply mean a subvariety $\mathcal{F}\subset \rat$ defined over $K$.  When no confusion can arise, we will identify the points in $\mathcal{F}(\kbar)$ with the associated rational functions.  We call the family {\it simple} if there are only finitely many elements in $\mathcal{F}(\Kbar)$ in any given conjugacy class of rational functions (equivalently, the induced map $\langle\cdot\rangle|_{\mathcal{F}}:\mathcal{F}\to \M_d$ is quasi-finite, i.e., has finite fibers).  
We give an example to show the necessity of the assumption that $\mathcal{F}$ be simple in Conjecture \ref{cmain}.

\begin{example}
\label{ex1}
Consider the family of quadratic polynomials 
\begin{equation*}
\phi_t(x)=(x+t)^2-t=x^2+2tx+t^2-t.  
\end{equation*}
We have $\phi_t^N(x)=(x+t)^{2^N}-t$.  So $\phi_{2^{2^N}}^N\left(2-2^{2^N}\right)=0$ and $|{\rm Preim}(\phi_t,0,\mathbb{Q})|$ is unbounded for $t\in \mathbb{Q}$ (in fact, even for $t\in \mathbb{Z}$).  Note that if $L_t(x)=x+t$ and $L_t^{-1}(x)=x-t$, then $L_t\circ \phi_t\circ L_t^{-1}(x)=x^2$.  So every rational function $\phi_t$ lies in the same conjugacy class.
\end{example}

To further explain this example, note that if $L$ is a linear fractional transformation over $K$, we have a bijection between the set $\pr$ and the set ${\rm Preim}(L\circ \phi\circ L^{-1},L(a),K)$.  So in studying $\pr$, it makes sense to define an equivalence relation $(\phi,a,K)\sim (\phi',a',K)$ if $\phi'=L\circ \phi\circ L^{-1}$ and $a'=L(a)$ for some linear fractional transformation $L$ defined over $K$.  From this point of view, if $\phi_t$ is as in Example \ref{ex1}, we see that for any fixed $a\in K$, the set $\{(\phi_t,a,\mathbb{Q})\mid t\in \mathbb{Q}\}$ is equivalent to the set $\{(x^2,a',\mathbb{Q})\mid a'\in \mathbb{Q}\}$.  So in looking at ${\rm Preim}(\phi_t,0,\mathbb{Q})$ for the family of functions $\phi_t$, in some sense, we haven't really fixed the element $a$ at all.  Example \ref{ex1}  is just another realization of the fact that $|\pr|$ is unbounded for $a\in K$.

We note that, in particular, Conjecture \ref{cmain} does not hold if we take $\mathcal{F}=\rat$ to be the family of all rational functions of a given degree $d\geq 2$.

Recall that the set of ($K$-rational) preperiodic points is the set
\begin{equation*}
\prp=\{x_0\in \mathbb{P}^1(K)\mid \phi^i(x_0)=\phi^j(x_0) \text{ for some } i>j\geq 0\}.
\end{equation*}
For any rational function $\phi$ of degree $d\geq 2$ and any number field $K$, the set $\prp$ is known to be finite.
In the same vein as Conjecture \ref{cmain}, Morton and Silverman \cite{MS} have made the following uniform boundedness conjecture for the cardinality of $\prp$:
\begin{conjecture}[Morton, Silverman]
\label{cMS}
Let $K$ be a number field of degree $D$ over $\mathbb{Q}$ and let $\phi\in K(x)$ be a rational function of degree $d\geq 2$.  Then there exists an integer $\kappa(D,d)$ such that
\begin{equation*}
|\prp|\leq \kappa(D,d).
\end{equation*}
\end{conjecture}
More generally, they have made a similar conjecture for endomorphisms of higher-dimensional projective space $\mathbb{P}^n$.

Also closely related to Conjecture \ref{cmain} is a conjecture of Silverman \cite[Conj.~4.98]{Sil} giving a lower bound for the canonical height of a wandering (i.e., non-preperiodic) rational point.
\begin{conjecture}[Silverman]
\label{cS}
Fix an embedding of $\M_d$ in projective space and let $h_{\M_d}$ denote the associated height function.  Let $K$ be a number field and $d\geq 2$ an integer.  Then there is a positive real number $\kappa(K,d)$ such that for all rational maps $\phi\in K(x)$ of degree $d$ and all wandering points $P\in \mathbb{P}^1(K)$,
\begin{equation*}
\hat{h}_\phi(P)\geq \kappa(K,d)\max\{\log \N_{K/\mathbb{Q}}\mathfrak{R}_\phi, h_{\M_d}(\langle \phi\rangle)\}.
\end{equation*}
\end{conjecture}
Here $\hat{h}_\phi$ is the canonical height associated to $\phi$ and $\mathfrak{R}_\phi$ is the minimal resultant of $\phi$ (see \cite{Sil} for definitions and properties).

We now relate Conjectures \ref{cMS} and \ref{cS} to Conjecture \ref{cmain}.

\begin{theorem}
\label{tcon}
Assume Conjectures \ref{cMS} and \ref{cS}.  Let $\mathcal{F}$ be a simple family of rational functions of degree $d\geq 2$ defined over a number field $k$.  Let $a \in \mathbb{P}^1(k)$ and let $K\supset k$ be a number field.  Then there exists an integer $\kappa(\mathcal{F},K,a)$ such that
\begin{equation}
\label{cineq}
|\pr|\leq\kappa(\mathcal{F},K,a)
\end{equation}
for every rational function $\phi\in \mathcal{F}(K)$.
\end{theorem}

Our main tool in the proof of Theorem \ref{tcon} will be a result of Call and Silverman \cite[Theorem 3.1]{CS} on the canonical height in families of varieties.
\begin{proof}
Let $\mathcal{V}=\mathbb{P}^1\times \mathbb{P}^{2d+1}$, $T=\mathbb{P}^{2d+1}$, $\pi:\mathcal{V}\to T$ and $\psi:\mathcal{V}\to \mathbb{P}^1$ the natural projection maps, and $\phi:\mathcal{V}\to \mathcal{V}$ the rational map given by
\begin{equation*}
[x,y]\times [a_0,\ldots,a_d,b_0,\ldots,b_d]\mapsto \left[\sum_{i=0}^d a_ix^iy^{d-i},\sum_{i=0}^d b_ix^iy^{d-i}\right]\times [a_0,\ldots,a_d,b_0,\ldots,b_d].
\end{equation*}
Let $T_0=\rat\subset T=\mathbb{P}^{2d+1}$.  On each fiber $\mathcal{V}_t\cong \mathbb{P}^1$ above a point $t\in T_0$, the map $\phi_t=\phi|_{\mathcal{V}_t}:\mathbb{P}^1\to \mathbb{P}^1$ is just the endomorphism of $\mathbb{P}^1$ associated to the point $t\in T_0=\rat$.  Let  $\eta=\psi^*\mathcal{O}(1)\in \Pic(\mathcal{V})$ and let $\alpha=d$.  Note that the height function associated to $\eta$ on $\mathcal{V}$ is just $h_{\mathcal{V},\eta}((P, t))=h(\psi((P,t)))=h(P)$, where $(P,t)\in \mathcal{V}$ with $P\in \mathbb{P}^1$ and $t\in T$.  Let $h_T$ be the usual height on $\mathbb{P}^{2d+1}$.  With the above choices, Theorem 3.1 of \cite{CS} says exactly that there exist constants $c_1$ and $c_2$ such that
\begin{equation}
\label{cheq1}
|\hat{h}_{\phi_t}(P)-h_{\mathbb{P}^1}(P)|<c_1h_T(t)+c_2, \qquad \forall t\in T_0, \forall P\in \mathbb{P}^1.
\end{equation}

Let $\overline{\M}_d$ be a projective compactification of $\M_d$ and let $A$ be an ample divisor on $\overline{\M}_d$ with associated height $h_A$, which we'll assume is everywhere positive and bounded away from zero.  We now specialize to the family $\mathcal{F}$, showing that there exist constants $c_1'$ and $c_2'$, depending on $\mathcal{F}$ and $h_A$, such that
\begin{equation}
\label{cheq3}
|\hat{h}_{\phi_t}(P)-h_{\mathbb{P}^1}(P)|<c_1'h_A(\langle \phi_t\rangle)+c_2', \qquad \forall t\in \mathcal{F}, \forall P\in \mathbb{P}^1.
\end{equation}

By \eqref{cheq1}, it suffices to show that there exist constants $c_3$ and $c_4$ such that
\begin{equation}
\label{cheq2}
h_T(t)<c_3h_A(\langle \phi_t\rangle)+c_4, \qquad \forall t\in \mathcal{F}.
\end{equation}
Let $\overline{\mathcal{F}}$ be the Zariski closure of $\mathcal{F}$ in $T$.  The morphism $\mathcal{F}\to \M_d$, $t\mapsto \langle \phi_t\rangle$, induces a rational map $\overline{\mathcal{F}}\to \overline{\M}_d$.  By blowing up $\overline{\mathcal{F}}$ we can find a projective variety $X$ such that there is a morphism $X\to \overline{\mathcal{F}}$ which is an isomorphism above $\mathcal{F}$ and the morphism $\mathcal{F}\to \M_d$ extends to a morphism $\sigma:X\to \overline{\M}_d$.  Since the family $\mathcal{F}$ is simple, it follows from the definitions that the morphism $\sigma$ is generically finite.  Thus $\sigma^*A$ is a big divisor on $X$.  Identifying $\mathcal{F}\subset X$ (as well as $\mathcal{F}\subset T$), it then follows from standard properties of heights (e.g., \cite[Prop. 1.2.9(h)]{V}) that since $\sigma^*A$ is big,
\begin{equation*}
h_T(t)\ll h_{\sigma^*A}(t)+O(1)=h_A(\langle\phi_t\rangle)+O(1),
\end{equation*}
for all $t\in \mathcal{F}\setminus \mathcal{F}'$, where $\mathcal{F}'$ is some Zariski-closed subset of $\mathcal{F}$.  Since every irreducible component of $\mathcal{F}'$ is again a simple family, by repeated application of the above argument we eventually arrive at \eqref{cheq2}.

Let $\phi\in \mathcal{F}(K)$.  We now consider two cases.  Suppose first that $a$ is preperiodic for $\phi$.  Then every element of $\pr$ is preperiodic and 
\begin{equation*}
\pr\subset \prp.
\end{equation*}
By Conjecture \ref{cMS}, there is an integer $\kappa_1(D,d)$ such that 
\begin{equation*}
|\pr|\leq \kappa_1(D,d).
\end{equation*}
Suppose now that $a$ is a wandering point of $\phi$ and $\phi^N(x_0)=a$ for some $x_0\in \mathbb{P}^1(K)$ and some positive integer $N$.  Then by Conjecture \ref{cS}, there is a positive real number $\kappa_2(K,d)$ such that
\begin{equation*}
\hat{h}(x_0)\geq \kappa_2(K,d) h_A(\langle \phi\rangle).
\end{equation*}
Combining this with \eqref{cheq3} and the fact that $d^N\hat{h}(x_0)=\hat{h}(a)$, we find
\begin{equation}
\label{dN}
d^N\kappa_2(K,d) h_A(\langle \phi\rangle)\leq d^N\hat{h}(x_0)=\hat{h}(a)<h(a)+c_1'h_A(\langle \phi\rangle)+c_2'.
\end{equation}
This bounds $N$ and shows that $|\pr|\leq \kappa_3(\mathcal{F},K,a)$ for some integer $\kappa_3(\mathcal{F},K,a)$.  Now the theorem follows taking $\kappa(\mathcal{F},K,a)=\max \{\kappa_1(D,d),\kappa_3(\mathcal{F},K,a)\}$.
\end{proof}

Actually, the proof of Theorem \ref{tcon} gives something somewhat stronger.  First, an examination of the proof shows that there should exist constants $c_1$ and $c_2$, depending on $\mathcal{F}$ and $K$, such that one can take $\kappa(\mathcal{F},K,a)=c_1h(a)+c_2$ in \eqref{cineq}.  Assuming a strong form of the Mordell conjecture, Faber \cite{Fab3} showed a similar result for the family of quadratic polynomials $x^2+t$, $t\in K$.

Second, note that for a simple family $\mathcal{F}$ of rational functions of degree $d$ defined over a number field $K$ and any real constant $c$, we have $h_A(\langle \phi\rangle)<c$ for only finitely many $\phi\in \mathcal{F}(K)$.  Then inequality \eqref{dN} and the proof of Theorem \ref{tcon} show that there exists a constant $\kappa(\mathcal{F},K)$ such that for any $a\in \mathbb{P}^1(k)$,
\begin{equation*}
|\pr|\leq\kappa(\mathcal{F},K)
\end{equation*}
for all but finitely many rational functions $\phi\in \mathcal{F}(K)$.  Thus, if we ignore finitely many rational functions in $\mathcal{F}(K)$ (depending on $a$), there should also be a bound independent of $a$.

Finally, one can also allow $a$ to vary with the rational function.  Let us call a subvariety $\mathcal{G}\subset \rat\times \mathbb{P}^1$ defined over $k$ a {\it $k$-family of rational functions of degree $d$ with a marked point}.  We will view a point of $\mathcal{G}$ as a pair $(\phi,a_\phi)$ where $\phi$ is a rational function and $a_\phi\in \mathbb{P}^1$.  We will call the family {\it simple} if the natural map $\mathcal{G}\to \rat\to \M_d$ has finite fibers (equivalently, there are only finitely many pairs $(\phi,a_\phi)\in \mathcal{G}(\kbar)$ with $\phi$ in a given conjugacy class of rational functions).  Theorem \ref{tcon} extends in the obvious way to simple families of rational functions with a marked point, with little change in the proof.  This suggests that Conjecture \ref{cmain} might hold in the following stronger form:
\begin{conjecture}
\label{cmain2}
Let $\mathcal{G}$ be a simple $k$-family of rational functions of degree $d\geq 2$ with a marked point.  Let $K\supset k$ be a number field of degree $D$ over $\mathbb{Q}$.  Then there exists an integer $\kappa(\mathcal{G},D)$ such that
\begin{equation*}
|{\rm Preim}(\phi,a_\phi,K)|\leq\kappa(\mathcal{G},D)
\end{equation*}
for every $(\phi,a_\phi)\in \mathcal{G}(K)$.
\end{conjecture}
Some further (weak) evidence for this conjecture is given in Theorem \ref{m2}.

Special cases of Conjecture \ref{cmain} are closely related to (known) results about uniform boundedness for torsion on elliptic curves.  We recall Merel's fundamental result (which built on earlier work of Mazur and Kamienny):
\begin{theorem}[Merel \cite{Merel}]
There exists an integer $\kappa(D)$ such that for any elliptic curve $E$ over a number field $K$ of degree $D$ over $\mathbb{Q}$, we have $|E(K)_{\tors}|\leq \kappa(D)$.
\end{theorem}

A weak version of Merel's theorem is related to Conjecture \ref{cmain} in the following way.  If $E$ is an elliptic curve in Weierstrass form, $a=\infty$, and $\phi=\phi_{E,m}$, $m>1$, is the rational function expressing $x(mP)$ in terms of $x(P)$, where $x(P)$ is the $x$-coordinate of a point $P\in E$, then the set $\pr$ contains the set of $x$-coordinates of points in the set $E[m^{\infty}](K)$, i.e., the set of all $K$-rational $m^N$-torsion points of $E$, $N\geq 1$.  As is well-known, the set of isomorphism classes of elliptic curves (over $\Qbar$) can be parametrized by $\mathbb{A}^1$, and such a parametrization gives rise, for a given $m>1$, to a one-dimensional simple family of rational functions $\phi_{E,m}$.  Then Conjecture \ref{cmain} easily implies (the known result) that there is an integer $\kappa(m,D)$ such that for any elliptic curve $E$ over a number field $K$ of degree $D$ over $\mathbb{Q}$, we have $|E[m^{\infty}](K)|\leq \kappa(m,D)$.  For a fixed number field $K$, this result on torsion was first proven by Manin \cite{Man}.  A higher-dimensional version of Conjecture \ref{cmain} would similarly have consequences for torsion on higher-dimensional abelian varieties (see the constructions in \cite{Fak}).  

The more usual connection between torsion on elliptic curves and arithmetic dynamics is through the set of preperiodic points.  The set of preperiodic points of $\phi_{E,m}$, for any $m>1$, corresponds exactly to the set of torsion points of $E$.  So taking $m=2$, the $d=4$ case of Morton and Silverman's conjecture (Conjecture \ref{cMS}) is sufficient to imply Merel's theorem.

\section{Unit equations and the proof of Theorem \ref{m1}}

Before proving Theorem \ref{m1} we recall a fundamental result of Evertse \cite{Ev} giving a bound for the number of ``non-degenerate" solutions to an $n$-term unit equation.

\label{secunit}
\begin{theorem}[Evertse \cite{Ev}]
\label{Evert}
Let $K$ be a number field and let $S$ be a finite set of places of $K$ containing the archimedean places.  Let $c_1,\ldots,c_n\in K^*$.  Suppose that $S$ has cardinality $s$.  Then the equation
\begin{align}
\label{ue}
c_1u_1+\cdots+c_nu_n=1 \quad &\text{ in $u_1,\ldots,u_n\in \co_{K,S}^*$ with}\\
& \sum_{i\in I}c_iu_i\neq 0 \text{ for each nonempty subset $I\subset \{1,\ldots,n\}$}\notag
\end{align}
has at most $(2^{35}n^2)^{n^3s}$ solutions.
\end{theorem}

\begin{proof}[Proof of Theorem \ref{m1}]
It suffices to prove the last part of the theorem.  Indeed, if $\phi=x^d$, then $|{\rm Preim}(\phi,a,K)|\leq \kappa''(d,D,a)$ for some $\kappa''(d,D,a)$.  Otherwise, if $\phi\neq x^d$ with $a_1,\ldots,a_{d-1},b_1,\ldots,b_{d-1}\in \co_{K,S}$ as in the theorem, then for some finite set of places $T$ of $K$, $a\in \co_{K,T}^*$, and so $|{\rm Preim}(\phi,a,K)|\leq \kappa'(d,|S\cup T|)$.  Then taking 
\begin{equation*}
\kappa(d,a,|S|)=\max \{\kappa''(d,2|S|,a),\kappa'(d,|S|+|T|)\} 
\end{equation*}
gives the first part of the theorem (note that $|T|$ depends only on $a$ and $D\leq 2|S|$).

We now assume that $a\in \co_{k,S}^*$.  Set $x_0=a$.  Let $x_1,\ldots, x_N$ satisfy $\phi(x_i)=x_{i-1}$ for $i=1,\ldots, N$ with all of the $x_i$ distinct.  Explicitly, $x_i$ satisfies the equation
\begin{equation}
\label{xeq}
x_i^d+a_{d-1}x_i^{d-1}+\cdots+a_1x_i-(b_{d-1}x_i^{d-1}+b_{d-2}x_i^{d-2}+\cdots+b_1x_i+1)x_{i-1}=0,
\end{equation}
$i=1,\ldots, N$.  Note that $\phi^i(x_i)=a$, $i=1,\ldots, N$.  Suppose that $x_i$ is $K$-rational for $i=1,\ldots, N$.  From \eqref{xeq} and induction, it is immediate that $x_i\in \co_{k,S}^*$ for $i=0,\ldots, N$.  Let $A$ be the matrix
\begin{equation*}
\tiny
\left[ \begin{array}{ccccccccc}
x_1^d-x_0 & x_1^{d-1} & x_1^{d-2} & \cdots & x_1 & x_1^{d-1}x_0 & x_1^{d-2}x_0 &\cdots & x_1x_0\\
x_2^d-x_1 & x_2^{d-1} & x_2^{d-2} & \cdots & x_2 & x_2^{d-1}x_1 & x_2^{d-2}x_1 &\cdots & x_2x_1\\
&\vdots\\
x_{2d-1}^d-x_{2d-2} & x_{2d-1}^{d-1} & x_{2d-1}^{d-2} & \cdots & x_{2d-1} & x_{2d-1}^{d-1}x_{2d-2} & x_{2d-1}^{d-2}x_{2d-2} &\cdots & x_{2d-1}x_{2d-2}\\
\end{array} \right]
\end{equation*}
and 
\begin{equation*}
\mathbf{v}=\left[ \begin{array}{ccccccc}
1 & a_{d-1} & \cdots & a_1 & -b_{d-1} & \cdots & -b_1
\end{array}\right]^T.  
\end{equation*}
Then we can rewrite the system of equations \eqref{xeq}, $i=1,\ldots, 2d-1$, as
\begin{equation*}
A\mathbf{v}=\mathbf{0}.
\end{equation*}
Since $\mathbf{v}$ is nontrivial, we have $\det A=0$.  Thus, there exists a polynomial $P(z_1,\ldots,z_{2d})$ in $\mathbb{Z}[z_1,\ldots,z_{2d}]$ such that $P(x_0,\ldots,x_{2d-1})=0$.  By the same argument, we see that $P(x_i,x_{i+1},\ldots,x_{2d-1+i})=0$ for any positive integer $i$ with $2d-1+i\leq N$.  Note that the polynomial $P$ depends only on the degree $d$.  Let $M_1,\ldots, M_m$ be the monomials appearing in $P(z_1,\ldots,z_{2d})$ and let 
\begin{equation*}
P(z_1,\ldots,z_{2d})=\sum_{j=1}^m c_jM_j(z_1,\ldots,z_{2d}), 
\end{equation*}
where $c_i\in \mathbb{Z}$ are the coefficients.  From our previous remarks, $M_j(x_i,\ldots,x_{2d-1+i})$ is an $S$-unit for all $i$ and $j$ and
\begin{equation*}
P(x_i,\ldots,x_{2d-1+i})=\sum_{j=1}^m c_jM_j(x_i,\ldots,x_{2d-1+i})=0
\end{equation*}
for all $i$.  So for $i=1,\ldots, N-2d+1$, we obtain a solution to the $S$-unit equation
\begin{equation*}
c_1u_1+\cdots+c_mu_m=0, \qquad u_1,\ldots,u_m\in \co_{k,S}^*,
\end{equation*}
where 
\begin{equation*}
(u_1,\ldots,u_m)=(M_1(x_i,\ldots,x_{2d-1+i}),\ldots,M_m(x_i,\ldots,x_{2d-1+i})).  
\end{equation*}

For some index set $J\subset \{1,\ldots, m\}$, there are at least $\frac{N-2d+1}{2^m}$ values of $i$ such that setting $u_j=M_j(x_i,\ldots,x_{2d-1+i})$ gives a solution to
\begin{equation*}
\sum_{j\in J}c_ju_j=0,
\end{equation*}
where $\sum_{j\in J'}c_ju_j\neq 0$ for every nonempty proper subset $J'\subset J$.  Dividing by any element $-c_ju_j$, $j\in J$, yields a solution to an equation as in  \eqref{ue}.  We now prove a lemma that bounds the number of distinct solutions to \eqref{ue} that are thus obtained.

\begin{lemma}
\label{lemman}
Let $j,j'\in \{1,\ldots,m\}$, $j\neq j'$.  Let $\alpha\in k^*$.  Let $\delta$ be the degree of the rational function 
\begin{equation*}
\frac{M_j(\phi^{2d-1}(x),\ldots,\phi(x),x)}{M_{j'}(\phi^{2d-1}(x),\ldots,\phi(x),x)}.
\end{equation*}
The equation
\begin{equation*}
M_j(x_i,\ldots,x_{2d-1+i})=\alpha M_{j'}(x_i,\ldots,x_{2d-1+i})
\end{equation*}
has at most $\delta$ distinct solutions $i\in \{1,\ldots,N-2d+1\}$.
\end{lemma}
\begin{proof}
Each solution $i$ gives a solution $x=x_{2d-1+i}$ to
\begin{equation*}
\frac{M_j(\phi^{2d-1}(x),\ldots,\phi(x),x)}{M_{j'}(\phi^{2d-1}(x),\ldots,\phi(x),x)}=\alpha.
\end{equation*}
If there are more than $\delta$ solutions in $i$ to 
\begin{equation*}
M_j(x_i,\ldots,x_{2d-1+i})=\alpha M_{j'}(x_i,\ldots,x_{2d-1+i}), 
\end{equation*}
then since the $x_i$ are distinct, from the definition of $\delta$ we must have that $\delta=0$ and 
\begin{equation}
\label{mjid}
M_j(\phi^{2d-1}(x),\ldots,\phi(x),x)=\alpha M_{j'}(\phi^{2d-1}(x),\ldots,\phi(x),x)
\end{equation}
identically.  We now show that this is impossible.  We may assume, after cancelling, that no variable appears nontrivially in both $M_j$ and $M_j'$ and (after possibly interchanging $j$ and $j'$) that for some $k\in \{1,\ldots,2d\}$, $z_k$ appears in $M_j(z_1,\ldots,z_{2d})$ and $z_l$ doesn't appear in $M_{j'}(z_1,\ldots,z_{2d})$ for $l\leq k$.  Since $\phi\neq x^d$, there exists $y\in \Kbar$ such that $\phi^{2d-k}(y)=0$ and $\phi^{2d-k'}(y)\neq 0$ for $2d\geq k'> k$.  For such a $y$ we have 
\begin{align*}
M_j(\phi^{2d-1}(y),\ldots,\phi(y),y)=0,\\
M_{j'}(\phi^{2d-1}(y),\ldots,\phi(y),y)\neq 0, 
\end{align*}
contradicting the identity \eqref{mjid}.
\end{proof}

Let 
\begin{equation*}
\mu=\max \deg\frac{M_j(\phi^{2d-1}(x),\ldots,\phi(x),x)}{M_{j'}(\phi^{2d-1}(x),\ldots,\phi(x),x)},
\end{equation*}
where $j,j'\in \{1,\ldots, m\}, j\neq j'$.  Then from the above and Lemma \ref{lemman}, we obtain at least $\frac{N-2d+1}{\mu 2^m}$ distinct solutions to \eqref{ue}.  By Theorem \ref{Evert}, it follows that 
\begin{equation*}
\frac{N-2d+1}{\mu 2^m}\leq (2^{35}m^2)^{m^3s}, 
\end{equation*}
where $s=|S|$.  Since $A$ is a $(2d-1)\times (2d-1)$ matrix, we can trivially estimate $m\leq (2d)!$.  From the explicit form of $A$, we easily find that for any $j$ and $l$, 
\begin{equation*}
\max \deg_{z_l} M_j(z_1,\ldots,z_{2d-1})\leq d+1.  
\end{equation*}
Since $\deg \phi^{j}(x)=d^j$, this implies that 
\begin{equation*}
\mu\leq (2d-1)d^{(2d-1)(d+1)}\leq d^{3d^2}.  
\end{equation*}
This gives an explicit bound on $N$ in terms of $d$ and $s$.  If $N$ is maximally chosen such that $x_1,\ldots, x_N\in K$ are distinct and $\phi^i(x_i)=a$, then we have the inequality 
\begin{equation*}
|{\rm Preim}(\phi,a,K)|\leq \sum_{i=1}^Nd^i\leq d^{N+1}.  
\end{equation*}
Using crude estimates and the bounds above, we then easily find, for instance, that
\begin{equation*}
|{\rm Preim}(\phi,a,K)|\leq \exp(\exp(d^{20d}s)).
\end{equation*}

\end{proof}

\section{Runge's method and the proof of Theorem \ref{m2}}

In the next section we reduce Theorem \ref{m2} to a problem about integral points on certain affine curves.  In the subsequent section, by using Runge's method and a result of Baker (Theorem \ref{tB}), the problem on integral points is reduced to the study of the arithmetic-geometric structure at infinity of these affine curves.  

 Runge's method applies, roughly, to affine curves which have enough rational points at infinity relative to various arithmetic parameters.  A natural class of curves which may frequently satisfy this constraint are curves that parametrize various algebro-geometric objects of interest.  As the rationality of the points on such curves (including the points at infinity) frequently carries significant arithmetic meaning, one may hope that the points at infinity have some structure that permits the application of Runge's method (on a ``generic" affine curve over $\mathbb{Q}$, in contrast, one expects there to be a single $\Gal(\Qbar/\mathbb{Q})$-orbit of points at infinity).  Indeed, Runge's method has been recently applied by Bilu and Parent \cite{BP} with surprising success to certain modular curves arising in a conjecture of Serre on Galois representations associated to elliptic curves (see also \cite{BP2} for another application to modular curves).  Here, we will also apply Runge's method to certain ``modular curves" which arise in the study of $\pr$.  This will reduce Theorem \ref{m2} to the study of the structure at infinity of these curves, allowing us to avoid the more intricate and difficult questions regarding the genus and reducibility of such curves.

Finally, in the last section we prove the necessary facts about the points at infinity on the relevant affine curves.

\subsection{Reduction to a theorem on integral points}
\label{sred}


Let $K$ be a number field and let $\mathcal{F}$ be a one-parameter simple family of rational functions of the following form:
\begin{equation*}
\phi_t(x)=\frac{x^d+a_{d-1}(t)x^{d-1}+\cdots+a_1(t)x+a_0(t)}{b_{d-1}(t)x^{d-1}+\cdots+b_1(t)x+b_0(t)},
\end{equation*}
where $d\geq 2$ is an integer and $a_0,\ldots,a_{d-1},b_0,\ldots,b_{d-1}\in K[t]$ are polynomials over $K$.  Let us also fix a polynomial $a\in K[t]$.

We now introduce some notation (we will follow, mostly, the notation of \cite{Fab}).  For simplicity, we will notationally omit the dependence on the polynomials $a$, $a_0,\ldots,a_{d-1}, b_0,\ldots,b_{d-1}$ (which we take as fixed).  For $N$ a positive integer, let $Y^{\pre}(N)$ denote the algebraic set in $\mathbb{A}^2$ (with coordinates $x$ and $t$) defined by $\phi_t^N(x)-a(t)=0$ (that is, the polynomial equation obtained after clearing denominators).  For each $N$, we have a rational map $\delta_N:Y^{\pre}(N)\to Y^{\pre}(N-1)$ given by $(x,t)\mapsto (\phi_t(x),t)$.

Let $Y^{\pre}(N)(\co_{L,S})=Y^{\pre}(N)\cap\mathbb{A}^2(\co_{L,S})$ be the set of $S$-integral points of $Y^{\pre}(N)$.  With $a,a_0,\ldots,a_{d-1},b_0,\ldots,b_{d-1}$ and the definitions as above, consider the following two statements:

\begin{statementA}
Let $s$ be a positive integer.  There exists a finite set of affine plane curves $\mathcal{C}=\mathcal{C}(s)$ such that for any positive integer $N$, the set of integral points
\begin{equation*}
\bigcup_{\substack{L\supset K,S_L\\|S_L|<s}}(Y^{\pre}(N)\setminus \cup_{C\in\mathcal{C}}C)(\co_{L,S_L})
\end{equation*}
is finite, where $L$ ranges over all number fields containing $K$ and $S_L$ ranges over all sets of places of $L$ containing the archimedean places.  Furthermore, $\mathcal{C}$ and this finite set of integral points, for a given $N$, are both effectively computable.
\end{statementA}

\begin{statementB}
Suppose that $K$ is a number field of degree $D$ over $\mathbb{Q}$.  Let $t\in K$ with $\deg \phi_t>1$ and let $s(t)$ be the number of primes $\mathfrak{p}$ of $K$ for which $|t|_\mathfrak{p}>1$.  There exists an effectively computable integer $\kappa(D,s(t))$ such that
\begin{equation*}
|{\rm Preim}(\phi_t,a(t),K)|\leq \kappa(D,s(t)).
\end{equation*}
\end{statementB}





\begin{lemma}
Statement A implies Statement B.
\end{lemma}
\begin{proof}
Let $s$ be a positive integer.  Let $k$ be a number field and $S$ a finite set of places of $k$ such that $a,a_0,\ldots,a_{d-1},b_0,\ldots,b_{d-1}\in \co_{k,S}[t]$.  Let $s'=s+D(|S|+1)+1$.  Let $\mathcal{C}=\mathcal{C}(s')$ be as in Statement A.  By modifying $\mathcal{C}$, we can assume that every curve $C\in \mathcal{C}$ is a component of $Y^{\pre}(N)$ for some $N$ and that every irreducible component of $Y^{\pre}(1)$ is in $\mathcal{C}$.  Let $\mathcal{C}'$ be the finite set of affine curves $C'$ such that $C'$ is an irreducible component of $Y^{\pre}(N)$ for some $N$, $\delta_N(C')\in \mathcal{C}$, and $C'\not\in \mathcal{C}$.  It follows from Statement A that the set
\begin{equation*}
R=\bigcup_{C'\in \mathcal{C}'}\bigcup_{\substack{L\supset k,S_L\\|S_L|<s'}}C'(\co_{L,S_L})
\end{equation*}
is finite and effectively computable.  Let 
\begin{equation*}
T=T(s')=\{t\mid (x,t)\in R, \deg\phi_t>1\}, 
\end{equation*}
a finite set.  

Let $t_0\in K$ be an element satisfying $\deg \phi_{t_0}>1$ and for which there are exactly $s$ primes $\mathfrak{p}$ of $K$ with $|t_0|_\mathfrak{p}>1$.  Let $x_0\in K$ satisfy $\phi_{t_0}^N(x_0)=a(t_0)$ for some positive integer $N$.  First, we note that 
\begin{equation}
\label{seq}
(\phi_{t_0}^{N'}(x_0),t_0)\in Y^{\pre}(N-N')(\co_{K,S'}), \qquad 0\leq N'< N, 
\end{equation}
for some set of places $S'$ of $K$ satisfying $|S'|<s'$.  Indeed, $t_0\in \co_{K,S_1}$ for some finite set of places $S_1$ with $|S_1|\leq s+D$ (the $D$ term comes from the archimedean places of $S_1$).    Let $S_2$ be the set of places in $K$ lying above places in $S$.  Note that $|S_2|\leq D|S|$.  Let $S'=S_1\cup S_2$.  Then $a(t_0),a_i(t_0),b_i(t_0)\in \co_{K,S'}$ for all $i$, and it follows from the form of $\phi_{t_0}$ that $\phi_{t_0}^{N'}(x_0)\in \co_{K,S'}$ for all $0\leq N'< N$.  So \eqref{seq} holds and we have $|S'|<s'$.  

We claim that either $t_0\in T$ or $(x_0,t_0)\in C$ for some curve $C\in \mathcal{C}$.  Suppose that $t_0\not\in T$.  Let $N'\geq 0$ be the smallest integer such that $(\phi^{N'}_{t_0}(x_0),t_0)\in C$ for some curve $C\in \mathcal{C}$.  Since we have assumed that every component of $Y^{\pre}(1)$ is in $\mathcal{C}$, we have $N'\leq N-1$.  If $N'>0$, then from the definitions, $(\phi^{N'-1}_{t_0}(x_0),t_0)\in C'$, for some $C'\in \mathcal{C}'$.  Then from the above,  $(\phi^{N'-1}_{t_0}(x_0),t_0)\in C'(\co_{K,S'})\subset R$, which contradicts our assumption that $t_0\not\in T$.  Thus, we must have $N'=0$ and $(x_0,t_0)\in C$ for some curve $C\in \mathcal{C}$.  This proves our claim.

We now give the desired bound for ${\rm Preim}(\phi_{t_0},a(t_0),K)$.  Suppose first that $t_0\not\in T$.  Since for fixed $t_0$, there are at most $\deg C$ points $(x_0,t_0)$ on any curve $C\in \mathcal{C}$, we have in this case
\begin{equation*}
|{\rm Preim}(\phi_{t_0},a(t_0),K)|\leq \sum_{C\in \mathcal{C}}\deg C.
\end{equation*}

If $t_0\in T$, since $T$ is a finite set, we simply take
\begin{equation*}
\mu=\mu(s')=\max_{t\in T}|{\rm Preim}(\phi_{t},a(t),K)|
\end{equation*}
and then 
\begin{equation*}
|{\rm Preim}(\phi_{t_0},a(t_0),K)|\leq \mu.
\end{equation*}
We note that $\mu$ is effectively computable (for a naive algorithm, for fixed $t\in T$ just compute successive preimages of $a(t)$ under $\phi_t$ until at some stage none of the preimages are $K$-rational).

It follows that Statement B holds with 
\begin{equation*}
\kappa(D,s)=\max\left\{\mu(s'),\sum_{C\in \mathcal{C}(s')}\deg C\right\},
\end{equation*}
noting that $s'$ depended only on $D$ and $s$.
\end{proof}

Thus, to prove Theorem \ref{m2} it suffices to prove Statement A for the appropriate one-parameter families of rational functions.
\subsection{Runge's method}

The main tool that will be used to prove cases of Statement A is an effective method for determining integral points on certain affine curves that goes back to Runge \cite{Run}.  We will use a suitably general form of Runge's result, due to Bombieri \cite{Bomb} (see also \cite{Lev}).
\begin{theorem}[Runge, Bombieri]
\label{Run}
Let $C\subset \mathbb{A}^N$ be an affine curve defined over a number field $K$.  Let $\tilde{C}$ be a projective closure of $C$.  Let $r$ be the number of $\Gal(\Kbar/K)$-orbits of points in $\tilde{C}(\Kbar)\setminus C$ (the number of $\Gal(\Kbar/K)$-orbits of points of $C$ at infinity).  Then the set of integral points
\begin{equation*}
\bigcup_{\substack{L\supset K,S_L\\|S_L|<r}}C(\co_{L,S_L})
\end{equation*}
is finite and can be effectively determined, where $L$ ranges over all number fields containing $K$ and $S_L$ ranges over all sets of places of $L$ containing the archimedean places.
\end{theorem}

Let $a,a_0,\ldots,a_{d-1},b_0,\ldots,b_{d-1}\in K[t]$ be polynomials as in Section \ref{sred}.  In view of Theorem \ref{Run}, we will be interested in the structure of $Y^{\pre}(N)$ at infinity (for simplicity, we again suppress the dependence on the polynomials $a,a_i,b_i$ from the notation).  Since $Y^{\pre}(N)$ is not necessarily geometrically irreducible, more precisely, we will need to study the structure of the points at infinity on the irreducible components of $Y^{\pre}(N)$.  Let $Y^{\pre}(N,1),\ldots,Y^{\pre}(N,r_N)$ denote the irreducible components of $Y^{\pre}(N)$ over $\Kbar$, where $r_N$ is the number of such irreducible components.  Recall that we have a rational map $\delta_N:Y^{\pre}(N)\to Y^{\pre}(N-1)$.  If $i_N\in \{1,\ldots, r_N\}$, then this map induces a (dominant) rational map 
\begin{equation*}
\delta_{N,i_N}:Y^{\pre}(N,i_N)\to Y^{\pre}(N-1,i_{N-1}) 
\end{equation*}
for some integer $i_{N-1}\in \{1,\ldots,r_{N-1}\}$.  Continuing, we obtain a chain of maps
\begin{equation*}
Y^{\pre}(N,i_N)\stackrel{\delta_{N,i_N}}{\rightarrow} Y^{\pre}(N-1,i_{N-1})\stackrel{\delta_{N-1,i_{N-1}}}{\rightarrow}\cdots\stackrel{\delta_{2,i_{2}}}{\rightarrow} Y^{\pre}(1,i_{1}).
\end{equation*}
Let us say that a sequence $i_1,i_2,i_3,\ldots$ is {\it admissible} if there are induced maps $\delta_{N,i_N}:Y^{\pre}(N,i_N)\to Y^{\pre}(N-1,i_{N-1})$ as above for every positive integer $N\geq 2$.

Let $X^{\pre}(N,i)$ denote the completion of the normalization $Y'^{\pre}(N,i)$ of $Y^{\pre}(N,i)$, $i=1,\ldots, r_N$.  Let $Y^{\pre}(N,i)(\infty)=X^{\pre}(N,i)(\Kbar)\setminus Y'^{\pre}(N,i)$ be the set of points ``at infinity" of $Y^{\pre}(N,i)$ and let
\begin{equation*}
Y^{\pre}(N)(\infty)=\bigsqcup_{i=1}^{r_N}Y^{\pre}(N,i)(\infty)
\end{equation*}
be the set of points at infinity of $Y^{\pre}(N)$.  We let $\deg_x Y^{\pre}(N,i)$ denote the degree in $x$ of the equation defining $Y^{\pre}(N,i)$ in $\mathbb{A}^2$.  To prove Statement A, we show that it suffices to control the structure of the points at infinity of $Y^{\pre}(N,i)$.

\begin{statementA'}
There exists a function $f:\mathbb{N}\to \mathbb{R}$ with $\lim_{n\to \infty}f(n)=\infty$ such that the following holds for any positive integer $N$ and any $i\in \{1,\ldots,r_N\}$:  Let $L$ be the minimal field of definition of $Y^{\pre}(N,i)$.  Then $Y^{\pre}(N,i)$ has at least $f(\deg_x Y^{\pre}(N,i))$ $\Gal(\overline{L}/L)$-orbits of points at infinity.
\end{statementA'}



We will show that Statement A' implies Statement A.  For this, we will also need a result of Baker \cite{Bak} on canonical heights over function fields.  If $F$ is a function field and $\phi\in F(x)$, we say that $\phi$ is {\it isotrivial} if it is conjugate by a linear fractional transformation over $\bar{F}$ to a function $\phi'$ defined over the constant field of $\bar{F}$.  We refer to \cite{Bak} for the relevant definitions.
\begin{theorem}[Baker]
\label{tB}
Let $F$ be a function field and let $\phi\in F(x)$ be a rational map of degree $d\geq 2$. Assume that $\phi$ is not isotrivial. Then there exists $\epsilon > 0$ (depending on $F$ and $\phi$) such that the set
\begin{equation*}
\{P\in \mathbb{P}^1(F)\mid \hat{h}_{\phi,F}(P)\leq \epsilon\}
\end{equation*}
is finite.
\end{theorem}

\begin{corollary}
\label{cB}
Let $\Delta$ be a positive integer.  Then there are only finitely many distinct curves $Y^{\pre}(N,i)$ with $\deg_x Y^{\pre}(N,i)\leq \Delta$.
\end{corollary}
\begin{remark}
It may happen that a curve arises as $Y^{\pre}(N,i)$ for infinitely many $N$ and $i$.
\end{remark}
\begin{proof}
Let $N$ be a positive integer and let $i_N\in \{1,\ldots, r_N\}$.  Let 
\begin{equation*}
\pi_N:Y^{\pre}(N,i_N)\to \mathbb{P}^1
\end{equation*}
be the morphism obtained by projecting onto the $t$-coordinate.  We have 
\begin{equation*}
\deg_x Y^{\pre}(N,i_N)=\deg \pi_N.  
\end{equation*}
Let $i_j$, $j\in \mathbb{N}$, be an admissible sequence containing $i_N$ as the $N$th element.  Then $\pi_N$ can be decomposed into the sequence of maps
\begin{equation*}
Y^{\pre}(N,i_N)\stackrel{\delta_{N,i_N}}{\rightarrow} Y^{\pre}(N-1,i_{N-1})\stackrel{\delta_{N-1,i_{N-1}}}{\rightarrow}\cdots\stackrel{\delta_{2,i_{2}}}{\rightarrow} Y^{\pre}(1,i_{1})\stackrel{\pi_1}{\rightarrow}\mathbb{P}^1,
\end{equation*}
where $\pi_1$ is the projection onto the $t$-coordinate.  So 
\begin{equation*}
\deg_x Y^{\pre}(N,i_N)=\deg \pi_N=\deg\pi_1\prod_{j=2}^N \deg \delta_{j,i_j}.
\end{equation*}

It follows that if there were infinitely many distinct curves $Y^{\pre}(N,i)$ with $\deg_x Y^{\pre}(N,i)\leq \Delta$, then there must exist a positive integer $N_0$ and an admissible sequence of integers $i_j$ such that $\deg \delta_{j,i_j}=1$ for all $j\geq N_0$ and the curves $Y^{\pre}(j,i_j)$, $j\geq N_0$, are all distinct.  Let $F=\Kbar(Y^{\pre}(N_0,i_{N_0}))$, the function field of $Y^{\pre}(N_0,i_{N_0})$.  We view $\phi_t(x)\in F(x)$ and $a(t)\in \mathbb{P}^1(F)$ in the natural way.  Since $\mathcal{F}$ is assumed to be a simple family, it follows easily that $\phi_t$ is not isotrivial.  Let $\epsilon>0$ be as in Theorem \ref{tB}.


The fact that $\deg \delta_{j,i_j}=1$ for all $j\geq N_0$ means exactly that there exist elements $\alpha_j\in F$ such that $\phi^j(\alpha_j)=a(t)$, $j\geq N_0$.  Since the curves $Y^{\pre}(j,i_j)$, $j\geq N_0$, are all distinct, the elements $\alpha_j$, $j\geq N_0$, are all distinct.  By a basic property of the canonical height, we have $d^j\hat{h}_{\phi,F}(\alpha_j)=\hat{h}_{\phi,F}(a)$, $j\geq N_0$.  But this implies that there are infinitely many $\alpha_j\in F$ with $\hat{h}_{\phi,F}(\alpha_j)<\epsilon$, contradicting Theorem \ref{tB}.
\end{proof}

The finitely many curves in Corollary \ref{cB} can be effectively computed.  Indeed, given $\Delta$, by the corollary there exists a positive integer $M$ such that for $i=1,\ldots,r_M$, either $\deg_x Y^{\pre}(M,i)>\Delta$ or $Y^{\pre}(M,i)=Y^{\pre}(M',i')$ for some $M'<M$, $i'\in \{1,\ldots,r_{M'}\}$.  This easily implies that for all $N> M$, $i\in \{1,\ldots,r_{N}\}$, either $\deg_x Y^{\pre}(N,i)>\Delta$ or $Y^{\pre}(N,i)=Y^{\pre}(M',i')$ for some $M'<M$ and $i'\in \{1,\ldots,r_{M'}\}$.  Thus, to find every such curve with $\deg_x Y^{\pre}(N,i)\leq \Delta$, we simply compute $Y^{\pre}(N,i)$ for all $i$ and $N=1,2,\ldots$ until we encounter such an integer $M$.

\begin{theorem}
\label{keyt}
Statement A' implies Statement A (and hence Statement B).
\end{theorem}

\begin{proof}

Let the function $f$ be as in Statement A' and let $s$ be a positive integer.  Let $\Delta$ be a positive integer such that $f(n)\geq 2s^2$ if $n\geq \Delta$.  Let 
\begin{equation*}
\mathcal{C}=\mathcal{C}(s)=\{Y^{\pre}(N,i)\mid N\geq 1, i\in \{1,\ldots, r_N\}, \deg_x Y^{\pre}(N,i)< \Delta\}.
\end{equation*}
By Corollary \ref{cB}, $\mathcal{C}$ is a finite set.  Let $Y^{\pre}(N,i)$ be a curve not in $\mathcal{C}$.  Then we have $\deg_x Y^{\pre}(N,i)\geq \Delta$.  Let $M$ be the minimal field of definition of $Y^{\pre}(N,i)$.  We consider two cases.

First, suppose that $[M:\mathbb{Q}]\geq 2s$.  Let $L\supset K$ be a number field and $S_L$ a finite set of places of $L$ containing the archimedean places satisfying $|S_L|<s$.  Then in particular, $[L:\mathbb{Q}]< 2s$ and $L\not\supset M$.  It follows that there is an element $\sigma\in \Gal(\Qbar/\mathbb{Q})$ that fixes $L$ but not $M$.  Then we must have 
\begin{equation*}
Y^{\pre}(N,i)(\co_{L,S_L})\subset (Y^{\pre}(N,i)\cap \sigma(Y^{\pre}(N,i)))(\overline{M}),
\end{equation*}
a finite set.  Moreover, there are only finitely many distinct such curves $\sigma(Y^{\pre}(N,i))$, so the set $Y^{\pre}(N,i)(\co_{L,S_L})$ is contained in a finite set independent of $L$.  It follows that in this case,
\begin{equation*}
\bigcup_{\substack{L\supset K,S_L\\|S_L|<s}}Y^{\pre}(N,i)(\co_{L,S_L})
\end{equation*}
is finite (and effectively computable).

Suppose now that $[M:\mathbb{Q}]<2s$.  By our assumptions, $f(\deg_xY^{\pre}(N,i))\geq 2s^2$ and $Y^{\pre}(N,i)$ has at least $2s^2$ $\Gal(\overline{M}/M)$-orbits of points at infinity.  By Theorem \ref{Run}, the set
\begin{equation*}
\bigcup_{\substack{M'\supset M,S_{M'}\\|S_{M'}|<2s^2}}Y^{\pre}(N,i)(\co_{M',S_{M'}})
\end{equation*}
is finite and effectively computable.  Since $[M:\mathbb{Q}]<2s$, if $L\supset K$ and $S_L$ is a finite set of places of $L$, then setting $M'=LM$ we have $|S_{M'}|< 2s|S_L|$, where $S_{M'}$ is the set of places of $M'$ lying above places of $S_L$.  So we find that the set
\begin{equation*}
\bigcup_{\substack{L\supset K,S_L\\|S_L|<s}}Y^{\pre}(N,i)(\co_{L,S_L})\subset \bigcup_{\substack{M'\supset M,S_{M'}\\|S_{M'}|<2s^2}}Y^{\pre}(N,i)(\co_{M',S_{M'}})
\end{equation*}
is finite (and effectively computable).  This proves Statement A with the set $\mathcal{C}$ as given above.
\end{proof}

It follows that to prove Theorem \ref{m2} we need to study in depth the structure of $Y^{\pre}(N)$ at infinity.

\subsection{The structure of $Y^{\pre}(N)$ at infinity}

In this section we prove Theorem \ref{m2} by proving appropriate versions of Statement A' in various cases (Theorems \ref{tr4}, \ref{tr1}, \ref{tr2}, \ref{tr3}).  Our main tool will be a lemma giving a sufficient criterion for a rational point to split into two rational points in a quadratic extension of function fields.  If $C$ is a curve over $K$, $P\in C(\Kbar)$, $\alpha\in K(C)^*$, and the principal divisor associated to $\alpha$ is $\dv(\alpha)=\sum_{Q\in C}n_QQ$, then we let $\ord_P \alpha=n_P$.

\begin{lemma}
\label{flemma}
Let $C$ be a nonsingular projective curve over a number field $K$ and let $P\in C(K)$.  Let $\alpha, \beta\in K(C)^*$ be two rational functions such that
\begin{equation*}
2\ord_P\beta<\ord_P\alpha.
\end{equation*}
Let $y\in \overline{K(C)}$ satisfy $y^2=\beta^2+\alpha$.  Then either 
\begin{enumerate}
\item $y\in K(C)$ and $K(C)(y)=K(C)$ \\\\
or\\
\item $[\Kbar(C)(y):\Kbar(C)]=2$ and if $\phi:X\to C$ is the morphism of nonsingular projective curves corresponding to the function field extension $K(C)(y)$ of $K(C)$, then $\phi^{-1}(P)$ consists of two $K$-rational points in $X(K)$.
\end{enumerate}
\end{lemma}
\begin{proof}
If $K(C)(y)=K(C)$ then the theorem is trivial.  Suppose first that 
\begin{equation*}
[\Kbar(C)(y):\Kbar(C)]=2.  
\end{equation*}
Without loss of generality, we can assume that $\beta$ has a pole at $P$, i.e.,  $\ord_P\beta<0$ (multiply $\beta$ by a suitable function and $\alpha$ by the square of this function).  Since $2\ord_P\beta<\ord_P\alpha$, there exists a positive integer $n$ such that 
\begin{equation*}
2n\ord_P\beta<(n+1)\ord_P\alpha.  
\end{equation*}
Let $F(x)=\sum_{i=0}^n \binom{\frac{1}{2}}{i}x^n$, the power series for $\sqrt{1+x}$ around $x=0$ truncated to order $n$.  As is well-known, 
\begin{equation*}
(F(x))^2=1+x+O(x^{n+1}).  
\end{equation*}
Let $\psi_+, \psi_-\in K(X)=K(C)(y)$ be the rational functions 
\begin{align*}
\psi_+&=\beta F\left(\frac{\alpha}{\beta^2}\right)+y,\\
\psi_-&=\beta F\left(\frac{\alpha}{\beta^2}\right)-y.
\end{align*}
We have an identity
\begin{align*}
\psi_+\psi_-&=\beta^2\left(1+\frac{\alpha}{\beta^2}+O\left(\left(\frac{\alpha}{\beta^2}\right)^{n+1}\right)\right)-y^2\\
&=\beta^2O\left(\left(\frac{\alpha}{\beta^2}\right)^{n+1}\right).
\end{align*}
Since $2n\ord_P\beta<(n+1)\ord_P\alpha$, it follows that $\psi_+\psi_-$ has a zero at every point in $X$ lying above $P$.  However, $\psi_++\psi_-=2\beta F\left(\frac{\alpha}{\beta^2}\right)$ has a pole at every point in $X$ lying above $P$.  Let $\tau$ be the involution of $X$ corresponding to $y\mapsto -y$.  Note that $\tau$ gives a bijection between the zeros of $\psi_+$ and the zeros of $\psi_-$.  Let $Q\in X$ lie above $P$.  Then from the above, $Q$ must be a zero of exactly one of $\psi_+$ and $\psi_-$ and $Q$ must be a pole of exactly one of $\psi_+$ and $\psi_-$.  Thus, $\tau(Q)\neq Q$ and there are two points lying above $P$.  Moreover, since $\psi_+,\psi_-\in K(X)$, we must have that $Q$ and $\tau(Q)$ are $K$-rational.

Finally, suppose that $K(C)(y)=L(C)$ for some number field $L\supset K$.  It is easy to see that $y$ must have the form $y=\gamma\psi$, where $L=K(\gamma)$ and $\psi\in K(C)$.  We have $(\frac{y}{\beta})^2=1+\frac{\alpha}{\beta^2}$.  By our assumptions, $\frac{\alpha}{\beta^2}$ has a zero at $P$.  So 
\begin{equation*}
\left(\frac{y}{\beta}\right)^2(P)=\gamma^2\left(\frac{\psi}{\beta}(P)\right)^2=1.  
\end{equation*}
Since $\frac{\psi}{\beta}(P)\in K$, it follows immediately that $\gamma\in K$ and $L=K$.
\end{proof}



\subsubsection{The case $\phi_t=x^2+b(t)x+c(t)$}

The first part of Theorem \ref{m2} follows from Theorem \ref{keyt} and the following result.

\begin{theorem}
\label{tr4}
Let $K$ be a number field and let $a,b,c\in K[t]$ be polynomials with $b^2-4c-2b$ nonconstant.  There exists a number field $K'$ such that for any positive integer $N$ and any $i\in \{1,\ldots, r_N\}$, $Y^{\pre}(N,i)$ is defined over $K'$ and has at least $\frac{\deg_x Y^{\pre}(N,i)}{4(1+\deg (b^2-4c+4a))}$ $K'$-rational points at infinity.
\end{theorem}

\begin{proof}
Let $j_1,j_2,\ldots$ be an admissible sequence.  Let 
\begin{equation*}
m=[\log_2 (1+\deg (b^2-4c+4a))]+1.  
\end{equation*}
Suppose first that $\deg (b^2-4c-2b)$ is even.  Let $K'$ be a number field such that
\begin{enumerate}
\item  $Y^{\pre}(m,j_m)$ is defined over $K'$.
\item  Every point at infinity of $Y^{\pre}(m,j_m)$ is defined over $K'$.
\item  $\sqrt{A}\in K'$, where $A$ is the leading coefficient of $b^2-4c-2b$.
\end{enumerate}

We claim that for $n> m$, $Y^{\pre}(n,j_n)$ is defined over $K'$ and that either 
\begin{equation*}
[K'(Y^{\pre}(n,j_n)):K'(Y^{\pre}(n-1,j_{n-1}))]=1, 
\end{equation*}
or over any point $P$ at infinity, $P\in Y^{\pre}(n-1,j_{n-1})(\infty)$, there are two points of $Y^{\pre}(n,j_n)(\infty)$, both defined over $K'$.  We have the function field identity $K'(Y^{\pre}(n,j_n))=K'(t,x_1,\ldots,x_n)$, where each $x_i$ satisfies 
\begin{equation*}
x_i^2+b(t)x_i+c(t)-x_{i-1}=0
\end{equation*}
and we have set $x_0=a(t)$.  Alternatively, solving the relevant quadratic equation in $x_i$, we easily see that $K'(Y^{\pre}(n,j_n))=K'(t,y_1,\ldots,y_n)$, where $y_i$ satisfies 
\begin{align}
\label{yeq1}
&y_1^2=b^2-4c+4a\\
&y_i^2=b^2-4c-2b+2y_{i-1},\qquad i\geq 2.
\label{yeq2}
\end{align}
In particular, $K'(Y^{\pre}(n,j_n))=K'(Y^{\pre}(n-1,j_{n-1}))(y_n)$.  We have 
\begin{equation*}
b^2-4c-2b=(\sqrt{A}t^e)^2+f(t) 
\end{equation*}
for some positive integer $e$ and some polynomial $f(t)$ satisfying $\deg f<2e$.  Let $P\in Y^{\pre}(n-1,j_{n-1})(\infty)$.  Using Lemma \ref{flemma} with $\beta=\sqrt{A}t^e$ and $\alpha=f(t)+2y_{n-1}$, by induction, to prove the claim it suffices to show that $2\ord_P\beta<\ord_P \alpha$ if $n>m$ (viewing all functions as functions on $Y^{\pre}(n-1,j_{n-1})$).  Since $\ord_P t<0$, 
\begin{equation*}
2\ord_P\beta=2e\ord_Pt<\ord_P f 
\end{equation*}
and it suffices to show that 
\begin{equation*}
2e\ord_P t<\ord_P y_{n-1}
\end{equation*}
if $n>m$.  If $m=1$ and $n=2$ then this is trivial (in this case, from the definition of $m$, $y_1$ is constant).  Otherwise, $n\geq 3$ and from \eqref{yeq2},
\begin{align*}
\ord_P y_{n-1}&\geq \frac{1}{2}\min\{\ord_P(b^2-4c-2b), \ord_P y_{n-2}\}\\
&\geq \frac{1}{2}\min\{2e\ord_P t, \ord_P y_{n-2}\}.
\end{align*}
Applying this repeatedly we obtain
\begin{equation*}
\ord_P y_{n-1} \geq \frac{1}{2}\min\{2e\ord_P t, \frac{1}{2^{n-3}}\ord_P y_{1}\}.
\end{equation*}
Let $C=\deg (b^2-4c+4a)$.  Then from \eqref{yeq1}, $\ord_P y_1=\frac{C}{2}\ord_P t$.  So
\begin{equation*}
\ord_P y_{n-1} \geq \min\left\{e, \frac{C}{2^{n-1}}\right\}\ord_P t.
\end{equation*}
Since $n>m$, $\frac{C}{2^{n-1}}\leq 1$, and so we find that $2e\ord_P t<\ord_P y_{n-1}$ as desired.

We now prove the theorem in this case.  Let $n>m$.  Then it follows from the above that either:
\begin{enumerate}
\item 
\begin{align*}
[K'(Y^{\pre}(n,j_n)):K'(Y^{\pre}(n-1,j_{n-1}))]=1,\\
\deg_x Y^{\pre}(n,j_n)=\deg_x Y^{\pre}(n-1,j_{n-1}), 
\end{align*}
and $Y^{\pre}(n,j_n)$ and $Y^{\pre}(n-1,j_{n-1})$ have the same number of points at infinity (all $K'$-rational).\\\\
or\\
\item 
\begin{align*}
[K'(Y^{\pre}(n,j_n)):K'(Y^{\pre}(n-1,j_{n-1}))]=2,\\ 
\deg_x Y^{\pre}(n,j_n)=2\deg_x Y^{\pre}(n-1,j_{n-1}), 
\end{align*}
and the curve $Y^{\pre}(n,j_n)$ has twice the number of points at infinity as $Y^{\pre}(n-1,j_{n-1})$ (all $K'$-rational).
\end{enumerate}
We note also that 
\begin{equation*}
\deg_x Y^{\pre}(m,j_m)\leq 2^m\leq 2(1+\deg (b^2-4c+4a)).  
\end{equation*}
The previous two statements then immediately imply that $Y^{\pre}(N,i)$ has at least $\frac{\deg_x Y^{\pre}(N,i)}{2(1+\deg (b^2-4c+4a))}$ $K'$-rational points at infinity.

The proof in the remaining case where $\deg (b^2-4c-2b)$ is odd is almost identical to the proof in the even case except that instead of considering the curves $Y^{\pre}(n,j_n)$, we consider appropriate affine curves $Y'^{\pre}(n,j_n)$ associated to the function field $K'(Y^{\pre}(n,j_n))(\sqrt{t})$ (we need the leading term of $b^2-4c-2b$ to be a perfect square in the function field).  With minor modifications, the proof in the even case then goes through for $Y'^{\pre}(n,j_n)$ (and if $Y'^{\pre}(n,j_n)$ has $r$ $K'$-rational points at infinity then $Y^{\pre}(n,j_n)$ has at least $\frac{r}{2}$ $K'$-rational points at infinity).
\end{proof}

\subsubsection{The case $\phi_t=\frac{x^2+b(t)x}{c(t)x+1}$}

We now prove versions of Statement A' in the case $\phi_t=\frac{x^2+b(t)x}{c(t)x+1}$, depending on the form of $a,b,c\in k[t]$.  Taken together, Theorems \ref{tr1}, \ref{tr2}, and \ref{tr3} complete the proof of Theorem \ref{m2}.

Let $N$ be a positive integer and let $i\in \{1,\ldots, r_N\}$.  The projection of $Y^{\pre}(N,i)$ onto the $t$-coordinate induces a morphism $\pi_{N,i}:X^{\pre}(N,i)\to \mathbb{P}^1$.  Assume, for the moment, that $Y^{\pre}(N,i)$ is defined over $K$.  From the definitions, the morphism $\pi_{N,i}$ corresponds to an extension of $K(t)$ given by $K(t,x_1,\ldots, x_N)$, where each $x_l$ satisfies 
\begin{equation*}
x_l^2+(b-x_{l-1}c)x_l-x_{l-1}=0, 
\end{equation*}
and where we have set $x_0=a$.  Each intermediate extension $K(t,x_1,\ldots, x_l)$, $1\leq l<N$, corresponds to a curve $X^{\pre}(l,j_l)$ for some $j_l\in \{1,\ldots, r_l\}$ ($j_1,\ldots,j_N=i$ forms part of an admissible sequence).  Note that $K(t,x_1,\ldots, x_l)=K(t,x_1,\ldots, x_{l-1})(y_l)$, where 
\begin{equation*}
y_l^2=(b-x_{l-1}c)^2+4x_{l-1}.
\end{equation*}

\begin{theorem}
\label{tr1}
Suppose that $b$ and  $c$ are nonconstant.  Let $N$ be a positive integer and $i\in \{1,\ldots, r_N\}$.  Then there exists a quadratic extension $K'$ of $K$ such that $Y^{\pre}(N,i)$ is defined over $K'$ and has at least $\frac{1}{2}\deg_x Y^{\pre}(N,i)$ $K'$-rational points at infinity (and $Y^{\pre}(N)$ has at least $2^{N-1}$ $K'$-rational points at infinity).  If 
\begin{equation*}
\deg a+\deg c\neq \deg b, 
\end{equation*}
then $Y^{\pre}(N,i)$ is defined over $K$ and has $\deg_x Y^{\pre}(N,i)$ $K$-rational points at infinity (and $Y^{\pre}(N)$ has $2^N$ $K$-rational points at infinity).
\end{theorem}

\begin{proof}

Suppose first that $\deg a+\deg c\neq \deg b$.  In view of Lemma \ref{flemma} and the above, our result will follow if we show that for $1\leq l\leq N$, we have 
\begin{equation}
\label{eqb}
2\ord_P(b-x_{l-1}c)<\ord_Px_{l-1}
\end{equation}
for every point $P$ in $Y^{\pre}(l-1,j_{l-1})(\infty)$ (setting $Y^{\pre}(0)=\mathbb{A}^1$).

First, we claim that if $P\in Y^{\pre}(l-1,j_{l-1})(\infty)$ and $\ord_Px_{l-1}+\ord_Pc\neq \ord_Pb$ then $2\ord_P(b-x_{l-1}c)<\ord_Px_{l-1}$.  In this case, 
\begin{equation*}
\ord_P(b-x_{l-1}c)=\min\{\ord_P b, \ord_Px_{l-1}c\}.  
\end{equation*}
Since $\deg b, \deg c\neq 0$, we have $\ord_P b,\ord_P c<0$.  If $\ord_Px_{l-1}\geq 0$, the claim follows from the above since $\ord_P b<0$.  If $\ord_Px_{l-1}<0$, then 
\begin{equation*}
\ord_Px_{l-1}c<\ord_Px_{l-1}<0 
\end{equation*}
and the claim again follows.

We now show that \eqref{eqb} holds for $l\geq 1$ and any $P\in Y^{\pre}(l,j_{l})(\infty)$.  From the above, it suffices to show that for $l\geq 0$ and any $P\in Y^{\pre}(l,j_{l})(\infty)$ we have 
\begin{equation}
\label{ordeq}
\ord_Px_{l}+\ord_Pc\neq \ord_Pb.  
\end{equation}
The case $l=0$ is equivalent to $\deg b\neq \deg a+\deg c$.  Suppose now that for an $l\geq 0$ and any $P\in Y^{\pre}(l,j_{l})(\infty)$ that \eqref{ordeq} holds.  Let $Q\in Y^{\pre}(l+1,j_{l+1})(\infty)$, lying above $P\in Y^{\pre}(l,j_{l})(\infty)$.  We have the identity
\begin{equation}
\label{eqid}
x_{l+1}(x_{l+1}+b-x_{l}c)=x_{l}.
\end{equation}
By our assumptions and Lemma \ref{flemma}, we have $\ord_Q \psi=\ord_P\psi$ for any rational function $\psi$ on  $Y^{\pre}(l,j_{l})$ (viewing $\psi$ as a rational function on $Y^{\pre}(l+1,j_{l+1})$ on the left-hand side).  In particular, $\ord_Qx_{l}+\ord_Qc\neq \ord_Qb$ and so 
\begin{equation*}
\ord_Q (b-x_{l}c)=\min\{\ord_Qb,\ord_Q x_l+\ord_Qc\}<0, 
\end{equation*}
as $\ord_Qb,\ord_Qc<0$.  If $\ord_Qx_{l+1}<\ord_Q (b-x_{l}c)<0$, then \eqref{eqid} implies that $\ord_Qx_{l}=2\ord_Qx_{l+1}$.  But then 
\begin{equation*}
\ord_Qx_{l+1}>\ord_Q x_l>\ord_Q x_l+\ord_Qc\geq \ord_Q (b-x_{l}c), 
\end{equation*}
contradicting the assumption.  Thus, either 
\begin{equation*}
\ord_Q x_{l+1}=\ord_Q(b-x_{l}c), 
\end{equation*}
or 
\begin{equation*}
\ord_Q(x_{l+1}+b-x_{l}c)=\ord_Q(b-x_{l}c).  
\end{equation*}
In the latter case, by \eqref{eqid}, $\ord_Qx_{l+1}=\ord_Qx_l-\ord_Q(b-x_{l}c)$.  Combining everything, we find that either
\begin{equation*}
\ord_Qx_{l+1}=\min\{\ord_Qb,\ord_Q x_l+\ord_Qc\}
\end{equation*}
or
\begin{equation*}
\ord_Qx_{l+1}=\ord_Qx_l-\min\{\ord_Qb,\ord_Q x_l+\ord_Qc\}.
\end{equation*}
Since $\ord_Qb,\ord_Qc<0$, in the first case, 
\begin{equation*}
\ord_Qx_{l+1}\leq\ord_Qb<\ord_Qb-\ord_Qc, 
\end{equation*}
and in the second case, 
\begin{equation*}
\ord_Qx_{l+1}\geq \ord_Qx_l-(\ord_Q x_l+\ord_Qc)> \ord_Qb-\ord_Qc.  
\end{equation*}
Thus, $\ord_Qx_{l+1}+\ord_Qc\neq \ord_Qb$ and we are done by induction.

We now consider the case $\deg b=\deg a + \deg c$.  Suppose first that $a$ is constant.  If $b-ac$ is also constant, then $Y^{\pre}(1)$ splits into two curves $x=\alpha$ and $x=\beta$, each of which satisfies the conclusion of the theorem.  We are then reduced to the same problem with $a$ replaced by, say, $\alpha$.  If $\alpha\neq a$ then $b-\alpha c$ is nonconstant.  So we may reduce to considering the situation where $b-ac$ is nonconstant, at the expense of possibly replacing $K$ by a quadratic extension $K'=K(\alpha)$ of $K$.  If $b-ac$ is nonconstant, we have $2\ord_\infty(b-ac)<\ord_\infty a=0$.  From our earlier proof, we then obtain that for any point $P\in Y^{\pre}(1,j_1)(\infty)$, $\ord_P x_1=\pm \ord_\infty(b-ac)\neq 0$.  Since $\deg b=\deg c$, $\ord_Pb=\ord_Pc$ and so $\ord_P x_1+\ord_Pc\neq \ord_Pb$.  The same proof as above now shows that for $l\geq 0$ and any $P\in Y^{\pre}(l,j_l)(\infty)$, we have $2\ord_P(b-x_{l}c)<\ord_Px_{l}$, proving the theorem in this case.

Suppose now that $a$ is nonconstant.  Let $P=P_N\in Y^{\pre}(N,j_N)(\infty)$ and let $P_l$ be the image of $P_N$ in $Y^{\pre}(l,j_l)(\infty)$, $0\leq l<N$.  We claim that 
\begin{equation}
\label{eqc}
2\ord_{P_l}(b-x_{l}c)<\ord_{P_{l}}x_{l} 
\end{equation}
for all but at most one value of $l$, $0\leq l<N$.  Let $l_0$, $0\leq l_0<N$, be such that $2\ord_{P_{l_0}}(b-x_{l_0}c)\geq\ord_{P_{l_0}}x_{l_0}$.  Then by the contrapositive of what we have proved earlier, it follows that we must have 
\begin{equation*}
\ord_{P_l}x_l+\ord_{P_l}c=\ord_{P_l}b, \qquad 0\leq l\leq l_0.  
\end{equation*}
Working on $Y^{\pre}(N,j_N)$, this implies that 
\begin{equation*}
\ord_Px_l=\ord_Px_{l'}=\ord_Pb-\ord_Pc=\ord_Pa
\end{equation*}
for all $l,l'$, $0\leq l,l'\leq l_0$.  In particular, for any $l$, $0\leq l<l_0$, $\ord_Px_l=\ord_Px_{l+1}$.  Choose such an $l$.  Then \eqref{eqid} implies that 
\begin{equation*}
\ord_P (x_{l+1}+b-x_lc)=0.  
\end{equation*}
Since $\ord_P x_{l+1}=\ord_Pa< 0$, we must have 
\begin{equation*}
\ord_P x_l=\ord_P x_{l+1}=\ord_P (b-x_lc)<0.  
\end{equation*}
We then obtain that $2\ord_P(b-x_{l}c)<\ord_Px_{l}$, which implies that \eqref{eqc} holds for $l<l_0$.  It follows that there can exist at most one such value $l_0$.  This proves our claim that \eqref{eqc} holds for all but at most one value of $l$, $0\leq l<N$.  Combined with Lemma \ref{flemma}, this easily implies that there is a quadratic extension $K'$ of $K$ such that $Y^{\pre}(N,i)$ is defined over $K'$ and that $Y^{\pre}(N,i)$ has at least $\frac{1}{2}\deg_x Y^{\pre}(N,i)$ $K'$-rational points at infinity.  The statements about $Y^{\pre}(N)$ follow immediately.

\end{proof}

\begin{lemma}
\label{lpol}
The rational functions $x_l$ have poles only at points at infinity.
\end{lemma}
\begin{proof}
This is obvious for $x_0=a(t)$.  Suppose the lemma is true for $x_l$.  Then $b-cx_l$ has poles only at points at infinity.  Now it follows immediately from \eqref{eqid} that if $x_{l+1}$ had a pole at a finite point then $x_l$ would have a pole at a finite point, contradicting our assumptions.  So $x_{l+1}$ must have poles only at points at infinity and the lemma follows by induction.
\end{proof}

\begin{theorem}
\label{tr2}
Suppose that $b$ is constant and $c$ is nonconstant.  Let $L$ be the minimal field of definition of $Y^{\pre}(N,i)$.  Then $Y^{\pre}(N,i)$ has at least $1+\log_2 \deg_x Y^{\pre}(N,i)$ $\Gal(\Lbar/L)$-orbits of points at infinity.
\end{theorem}
\begin{proof}
We have $\deg_x Y^{\pre}(N,i)=[\Lbar(t,x_1,\ldots,x_N):\Lbar(t)]$.  Suppose first that $a$ is nonzero.  It is easily seen that $Y^{\pre}(1)$ is irreducible, $\deg_x Y^{\pre}(1)=2$, and $Y^{\pre}(1)$ has two $K$-rational points at infinity.  By induction, it suffices to show that if for some $l$, $1\leq l<N$, $[\Lbar(t,x_1,\ldots,x_{l+1}):\Lbar(t,x_1,\ldots,x_l)]=2$ then some $\Gal(\Lbar/L)$-orbit of $Y^{\pre}(l,j_l)(\infty)$ splits into two $\Gal(\Lbar/L)$-orbits in $Y^{\pre}(l+1,j_{l+1})(\infty)$.  We now show this.

It's easy to see from \eqref{eqid} and Lemma \ref{lpol} that under our assumptions $x_l$ can never be constant for any $l>0$.  Let $P\in Y^{\pre}(l,j_l)(\infty)$ be a pole of $x_l$.  Since $b$ is constant, we have $2\ord_P(b-x_{l}c)<\ord_Px_{l}$.  By Lemma \ref{flemma}, $P$ splits into two $L(P)$-rational points in $Y^{\pre}(l+1,j_{l+1})(\infty)$.  It follows that there are two distinct $\Gal(\Lbar/L)$-orbits in $Y^{\pre}(l+1,j_{l+1})(\infty)$ lying above the $\Gal(\Lbar/L)$-orbit containing $P$ in $Y^{\pre}(l,j_l)(\infty)$.

If $a=0$ then $Y^{\pre}(1)$ splits into two curves:  one defined by $x=0$ and one defined by $x=-b$.  For these curves the theorem is trivially true.  At the next step, we are reduced to considering the problem with the same $b$ and $c$ and with $a=0$ replaced by $a=-b$.  If $b=0$ then the only preimage curve is $x=0$ (and there's nothing more to prove).  Otherwise, if $b\neq 0$, the relevant case was already proved above.
\end{proof}

If $a$ and $c$ are both constant then it suffices to consider the case $b=t$.

\begin{theorem}
\label{tr3}
Suppose that $a$ and $c$ are constant and $b=t$.  Let $L$ be the minimal field of definition of $Y^{\pre}(N,i)$.  Then $Y^{\pre}(N,i)$ has at least $\log_2 \deg_x Y^{\pre}(N,i)$ $\Gal(\Lbar/L)$-orbits of points at infinity.
\end{theorem}
\begin{proof}
Suppose first that $a\neq 0$.  We first show by induction that $x_l$ has zeros only at infinity on $Y^{\pre}(l,j_l)$.  This is trivially true for $x_0=a$.  Suppose it is true for $x_l$.  Since $x_{l+1}+t-cx_l$ has poles only at infinity, it follows immediately from the identity \eqref{eqid} that any finite zero of $x_{l+1}$ would be a zero of $x_l$.  We note also that for $l>0$ (and $a\neq 0$), it is easy to see that $x_l$ is nonconstant.  Furthermore, $Y^{\pre}(1)$ is irreducible and has two $K$-rational points at infinity.

For any zero $P$ of $x_l$ we have $\ord_P(b-x_{l}c)=\ord_Pt<0$ and $\ord_Px_{l}>0$.  It follows that $2\ord_P(b-x_{l}c)<\ord_Px_{l}$.  Now the same proof as in Theorem~\ref{tr2} shows that if $1\leq l<N$ and $[\Lbar(Y^{\pre}(l+1,j_{l+1})):\Lbar(Y^{\pre}(l,j_l))]=2$, then there are two distinct $\Gal(\Lbar/L)$-orbits in $Y^{\pre}(l+1,j_{l+1})(\infty)$ lying above the $\Gal(\Lbar/L)$-orbit containing $P$ in $Y^{\pre}(l,j_l)(\infty)$.  This is sufficient to imply the theorem in this case.  Alternatively, one could directly deduce the case $a\neq 0$ here in Theorem \ref{m2} from Theorem \ref{m1} or from Theorem \ref{tr2} using an identity as in Remark \ref{rem7}.

Suppose now that $a=0$.  In this case, $Y^{\pre}(1)$ consists of the two curves defined by $x=0$ and $x=-t$.  Both curves satisfy the conclusion of the theorem.  The curve $x=0$ leads to nothing essentially new and it suffices to consider the curves in $Y^{\pre}(N)$ lying above the curve $x=-t$.  So suppose that for $l\geq 2$, $Y^{\pre}(l,j_l)$ lies above the curve $x=-t$ via the natural maps.  We first claim that 
\begin{equation*}
[\Kbar(Y^{\pre}(l+1,j_{l+1})):\Kbar(Y^{\pre}(l,j_l))]=2 
\end{equation*}
for all $l\geq 1$ (in particular, $Y^{\pre}(l,j_l)$ is defined over $K$ for all $l$).  This follows essentially from Eisenstein's criterion.  More precisely, there is a unique point on $Y^{\pre}(l+1,j_{l+1})$ lying above $t=0$, it is a ramification point of the map 
\begin{equation*}
Y^{\pre}(l+1,j_{l+1})\to Y^{\pre}(l,j_l), 
\end{equation*}
and $x_{l+1}$ has a simple zero at this point.  For $Y^{\pre}(2,j_{2})$ and $x_2$ this is clear from the defining equation
\begin{equation*}
x_2^2+(1+c)tx_2+t=0.
\end{equation*}
The general case follows easily by induction from the equation
\begin{equation*}
x_{l+1}^2+(t-cx_l)x_{l+1}-x_l=0.
\end{equation*}

We now claim that $x_l$ has zeros only at infinity and above $t=0$.  This is true for $x_1=-t$.  The general case follows by induction as $x_{l+1}(x_{l+1}+t-cx_l)=x_l$ and $x_{l+1}+t-cx_l$ has poles only at infinity, so as mentioned previously, it is immediate that any finite zero of $x_{l+1}$ must be a zero of $x_l$.

If $\deg x_l>1$, then since $x_l$ has only a simple zero at the unique point above $t=0$, there must be some point $P\in Y^{\pre}(l,j_l)(\infty)$ with $\ord_P x_l>0$.  For this point we have $2\ord_P(b-x_{l}c)<\ord_Px_{l}$ and it follows that the $\Gal(\Kbar/K)$-orbit containing $P$ must split into two distinct $\Gal(\Kbar/K)$-orbits in $Y^{\pre}(l+1,j_{l+1})$.

If $\deg x_l=1$ and $Y^{\pre}(l,j_l)$ has more than one point at infinity, then there must be a point $P$ at infinity with $\ord_Px_l=0$.  Again we have $2\ord_P(b-x_{l}c)<\ord_Px_{l}$ and it follows that the $\Gal(\Kbar/K)$-orbit containing $P$ must split into two distinct $\Gal(\Kbar/K)$-orbits in $Y^{\pre}(l+1,j_{l+1})$.

A simple computation shows that $Y^{\pre}(2,j_2)$ has two $K$-rational points at infinity unless $c=-1$.  When $c=-1$, we find that $Y^{\pre}(2,j_2)$ has a single $K$-rational point at infinity but $Y^{\pre}(3,j_3)$ has two $K$-rational points at infinity.

Combining all of the above statements completes the proof in the case $a=0$.
\end{proof}

\subsection*{Acknowledgments}
The author would like to thank Xander Faber for many helpful remarks on an earlier draft of the paper.

\bibliography{dyn}
\end{document}